\documentclass[11pt]{article}
\usepackage{geometry}
\usepackage{amsmath,amssymb,amsthm}
\usepackage{charter}
\usepackage{cite}
\usepackage{xcolor}
\usepackage[nottoc]{tocbibind} 
\usepackage[pagebackref=true]{hyperref}
\renewcommand*\backref[1]{\ifx#1\relax \else \mbox{\textcolor{gray}{Page #1}} \fi}

\hypersetup{
	colorlinks,%
	citecolor=blue,%
	filecolor=black,%
	urlcolor=black,%
	linkcolor=red
}

\newcommand\blfootnote[1]{%
	\begingroup
	\renewcommand\thefootnote{}\footnote{#1}%
	\addtocounter{footnote}{-1}%
	\endgroup
}

\newtheorem{theorem}{Theorem}[section]
\newtheorem{lemma}[theorem]{Lemma}
\newtheorem{proposition}[theorem]{Proposition}
\newtheorem{corollary}[theorem]{Corollary}
\theoremstyle{remark}
\newtheorem{remark}[theorem]{Remark}
\numberwithin{equation}{section}

\newcommand{\N}{\mathbb{N}}
\newcommand{\R}{\mathbb{R}}
\newcommand{\Rn}{\R^n}

\newcommand{\sn}{\mathbb{S}^{n-1}}
\newcommand{\K}{\mathcal{K}}
\newcommand{\Kn}{\K^n}

\newcommand{\hm}{\mathcal H} 

\newcommand{\fconvs}{{\mathrm{Conv}_{\mathrm{sc}}(\R^n)}} 
\newcommand{\fconvf}{{\mathrm{Conv}(\R^n; \R)}} 

\newcommand{\oZ}{\operatorname{Z}}
\newcommand{\oz}{\operatorname{z}}
\newcommand{\om}{\operatorname{m}}
\newcommand{\ot}{\operatorname{t}}
\newcommand{\dom}{\operatorname{dom}}
\newcommand{\epi}{\operatorname{epi}}
\newcommand{\vol}{\operatorname{vol}}
\newcommand{\oZZ}[2]{{\operatorname{V}}_{#1,#2}} 
\newcommand{\MA}{\mathrm{MA}} 

\newcommand{\sq}{\mathbin{\vcenter{\hbox{\rule{.3ex}{.3ex}}}}} 

\newcommand{\SO}{\operatorname{SO}}
\newcommand{\SOn}{\SO(n)}

\newcommand{\On}{\operatorname{O}(n)}
\newcommand{\Ot}{\operatorname{O}(2)}

\renewcommand{\d}{\,\mathrm{d}} 
\newcommand{\Hess}{{\operatorname{D}}^2} 
\newcommand{\ind}{{\mathbf{I}}} 

\begin{document}
    \title{A Klain--Schneider Theorem for Vector-Valued Valuations on Convex Functions}
    \author{Mohamed A.\ Mouamine and Fabian Mussnig}
    \date{}
    \maketitle
    
    \begin{abstract}		
        A functional analog of the Klain--Schneider theorem for vector-valued valuations on convex functions is established, providing a classification of continuous, translation covariant, simple valuations. Under additional rotation equivariance assumptions, an analytic counterpart of the moment vector is characterized alongside a new epi-translation invariant valuation. The former arises as the top-degree operator in a family of functional intrinsic moments, which are linked to functional intrinsic volumes through translations. The latter represents the top-degree operator in a class of Minkowski vectors, which are introduced in this article and which lack classical counterparts on convex bodies, as they vanish due to the Minkowski relations. Additional classification results are obtained for homogeneous valuations of extremal degrees.
        
    \end{abstract}
    
    {
        \hypersetup{linkcolor=black}
        \small
        \tableofcontents
    }
    
    \goodbreak
    
    \normalsize
        
    \section{Introduction}
    \paragraph{The Klain--Schneider Theorem}
    Let $\Kn$ denote the set of non-empty, compact, convex subsets of $\Rn$, whose elements we call \textit{convex bodies}. We say that a map $\oZ$ on $\Kn$ that takes values in an Abelian semigroup $\langle \mathbb{A},+\rangle$ is a \textit{valuation}, whenever
    \begin{equation*}
        \oZ(K \cup L) + \oZ(K\cap L)=\oZ(K)+\oZ(L)
    \end{equation*}
    for every $K,L\in\Kn$ such that also $K\cup L\in\Kn$. The probably most important example of a real-valued valuation is the \textit{$n$-dimensional volume} or \textit{Lebesgue measure}, denoted by $\vol_n$. For vector-valued valuations, taking values in $\Rn$, a similar role is played by the \textit{moment vector}
	\[
	m(K)=\int_K x\d x
	\]
	for $K\in\Kn$, where integration is understood with respect to the $n$-dimensional Lebesgue measure on $\Rn$.
    
    Starting with Dehn's solution of Hilbert's Third Problem, the systematic study of valuations on convex bodies has since become a cornerstone of convex geometry. In line with Felix Klein's Erlangen Program, particular attention has been given to classifying valuations that remain invariant under specific group actions (see, for example, \cite{alesker_annals_1999,alesker_faifman_2014,bernig_gafa_2009,hug_schneider_gafa_2014,ludwig_reitzner_2010,schuster_wannerer_jems_2018}). An early highlight of this line of research is Hadwiger's celebrated characterization of the \textit{intrinsic volumes} \cite{hadwiger}. About 30 years ago, Klain found a classification of even, continuous (w.r.t.\ the Hausdorff metric), translation invariant, simple valuations, which allowed for an elegant new proof of Hadwiger's characterization theorem \cite{klain_1995}. Here, a valuation $\oZ\colon\Kn\to\R$ is \textit{translation invariant} if $\oZ(K+x)=\oZ(K)$ for every $K\in\Kn$ and $x\in\Rn$, and it is \textit{simple} if it vanishes on convex bodies of dimension less than $n$. Shortly after Klain's work, Schneider could remove the assumption of evenness, leading to what is now known as the Klain--Schneider theorem \cite[Theorem 2]{schneider_simple}.
	
    \begin{theorem}[Klain--Schneider]
		\label{thm:klain_schneider}
		A map $\oZ\colon\Kn\to\R$ is a continuous, translation invariant, simple valuation, if and only if there exist a constant $c\in\R$ and an odd, continuous function $g\colon\sn\to\R$ such that
		\[
		\oZ(K)=c\, \vol_n(K)+\int_{\sn} g(z)\d S_{n-1}(K,z)
		\]
		for every $K\in\Kn$.
    \end{theorem}
    \goodbreak
    \noindent
    Here, $S_{n-1}(K,\cdot)$ is the \textit{surface area measure} of the convex body $K\in\Kn$. If $K$ is full-dimensional, then $S_{n-1}(K,\omega)$ is the $(n-1)$-dimensional Hausdorff measure of the boundary points of $K$ that have an outer unit normal in the Borel set $\omega\subseteq\sn$.
    
    \medskip
    
    The Klain--Schneider theorem became an important asset in the development of modern valuation theory. See, for example, \cite{alesker_annals_1999,alesker_gafa_2001,alesker_faifman_2014,bernig_hug_2018,boeroeczky_domokos_solanes_jfa_2021,klain_2001,li_ma_jfa_2017,parapatits_wannerer,schuster_wannerer_ajm_2015,solanes_wannerer}. One of its numerous consequences is the fact, that for $n\geq 2$, multiples of $\vol_n$ are the only continuous, translation and rotation invariant, simple valuations, where $\oZ\colon\Kn\to\R$ is \textit{rotation invariant} if $\oZ(\vartheta K)=\oZ(K)$ for every $K\in\Kn$ and $\vartheta\in\SOn$. From this fact together with the relation
    \[
    m(K+x)=m(K)+\vol_n(K)x
    \]
    for $K\in\Kn$ and $x\in\Rn$, one can obtain the following characterization of the moment vector. We remark that this result independently follows from \cite[Hauptsatz]{hadwiger_schneider} and refer to \cite{haberl_jems,ludwig_tams_2005,zeng_ma_tams} for further characterizations of the moment vector. We will use a lowercase $\oz$ for vector-valued valuations throughout this article.
	
    \begin{corollary}
        \label{cor:moment_vector}
        For $n\geq 2$, a map $\oz\colon\Kn\to\Rn$ is a continuous, translation covariant, rotation equivariant, simple valuation, if and only if there exists a constant $c\in\R$ such that $\oz(K)=c\,m(K)$ for every $K\in\Kn$.
    \end{corollary}
    \noindent
    Here, we say that $\oz\colon\Kn\to\Rn$ is \textit{translation covariant} if there exists an associated map $\oz^0\colon\Kn\to\R$ such that
    \[
    \oz(K+x)=\oz(K)+\oz^0(K)\,x
    \]
    for every $K\in\Kn$ and $x\in\Rn$. Furthermore, $\oz$ is \textit{rotation equivariant} if $\oz(\vartheta K)=\vartheta \oz(K)$ for every $K\in\Kn$ and $\vartheta\in\SOn$.
	
    \paragraph{Real-Valued Valuations on Convex Functions}
    Among the many extensions and generalizations of valuations are their extensions from convex bodies to various functions spaces \cite{alesker_adv_geom_2019,baryshnikov_ghrist_wright,colesanti_pagnini_tradacete_villanueva_jfa_2021,ludwig_adv_geom_2011,ludwig_ajm_2012,mussnig_lc,tsang_mink}. Of significant interest are convex functions which have been a particular focus of research during the last decade \cite{colesanti_ludwig_mussnig_1,colesanti_ludwig_mussnig_2,colesanti_ludwig_mussnig_4,colesanti_ludwig_mussnig_3,colesanti_ludwig_mussnig_8,colesanti_ludwig_mussnig_5,knoerr_support,knoerr_singular,knoerr_smooth,knoerr_ulivelli_math_ann_2024,knoerr_ulivelli_polynomial}. Let
    \[
    \fconvf=\{v\colon\Rn\to\R : v \text{ is convex}\}
    \]
    denote the space of (finite-valued) convex functions on $\Rn$. A map $\oZ\colon\fconvf\to\langle \mathbb{A},+\rangle$ is a \textit{valuation} if
    \[
    \oZ(v\vee w) + \oZ(v\wedge w) = \oZ(v)+\oZ(w)
    \]
    for every $v,w\in\fconvf$ such that also their pointwise maximum $v\vee w$ and pointwise minimum $v\wedge w$ are elements of $\fconvf$. In this setting, the second author, together with Colesanti and Ludwig, obtained a functional version of Hadwiger's theorem, which characterizes so-called \textit{functional intrinsic volumes} \cite{colesanti_ludwig_mussnig_5}. Shortly afterward, the same authors found a new, more efficient approach to this result, based on an extension of Klain's approach to Hadwiger's theorem \cite{colesanti_ludwig_mussnig_8}. A critical element developed for this new proof is a Klain--Schneider theorem for valuations on convex functions. In the following we say that a map $\oZ$ on $\fconvf$ is \textit{dually translation invariant} if
    \[
    \oZ(v+\ell)=\oZ(v)
    \]
    for every $v\in\fconvf$ and every linear functional $\ell\colon\Rn\to\R$. Furthermore, $\oZ$ is \textit{vertically translation invariant} if $\oZ(v+c)=\oZ(v)$ for every $v\in\fconvf$ and $c\in\R$. If $\oZ$ satisfies both of these properties, then it is called \textit{dually epi-translation invariant}, where this name comes from the fact that the composition of $\oZ$ with convex conjugation is invariant under the translation of epi-graphs (see Section~\ref{se:class_simple}). Lastly, $\oZ$ is \textit{dually simple} when $\oZ$ vanishes on functions that, up to the addition of linear functionals, depend on less than $n$ variables (in a suitable coordinate system). In the following result, which was obtained in \cite[Theorem 3.5]{colesanti_ludwig_mussnig_8}, continuity is understood with respect to \textit{epi-convergence} of functions, which on $\fconvf$ coincides with pointwise convergence.
	
    \begin{theorem}[Klain--Schneider theorem on $\fconvf$]
        \label{thm:klain_schneider_fconvf}
        A map $\oZ\colon\fconvf\to\R$ is a continuous, dually epi-translation invariant, dually simple valuation, if and only if there exists a function $\zeta\in C_c(\Rn)$ such that
        \[
        \oZ(v)=\int_{\Rn} \zeta(x) \d\MA(v;x)
        \]
        for every $v\in\fconvf$.
    \end{theorem}
    \noindent
    Here, $C_c(\Rn)$ denotes the set of continuous functions with compact support on $\Rn$, and $\MA(v;\cdot)$ is the \textit{Monge--Amp\`ere measure} associated with $v\in\fconvf$.	If in addition $v\in C^2(\Rn)$, then
    \begin{equation}
        \label{eq:ma_det}
        \d \MA(v;\cdot)=\det(\Hess v(x)) \d x,
    \end{equation}
    where $\det(\Hess v(x))$ denotes the determinant of the Hessian matrix of $v$ at $x\in\Rn$. As a function of $v$, equation \eqref{eq:ma_det} continuously extends to $\fconvf$ with respect to the topology induced by weak convergence of measures (see, for example, \cite[Proposition 2.6]{figalli}).  At this point, let us mention that recently Theorem~\ref{thm:klain_schneider_fconvf} inspired an analog of the Klain--Schneider theorem for zonal valuations on convex bodies \cite{brauner_hofstaetter_ortega-moreno_zonal}. See also \cite{knoerr_zonal}.
    
    \paragraph{Vector-Valued Valuations on Convex Functions}
    The focus of this paper lies on vector-valued valuations on $\fconvf$ that commute with rotations. We say that map $\oz\colon\fconvf\to\Rn$ is \textit{dually translation covariant} if there exists an associated map $\oz^0\colon\fconvf\to\R$ such that
    \[
    \oz(v+\langle x, \cdot \rangle) = \oz(v)+\oz^0(v)\, x
    \]
    for every $v\in\fconvf$ and $x\in\Rn$, where $\langle\cdot,\cdot\rangle$ is the usual scalar product on $\Rn$. In the following we associate with each $v\in\fconvf$ the \textit{Hessian measure} $\Theta_0(v;\cdot)$, which is the unique Borel measure on $\Rn\times\Rn$ such that $v\mapsto \Theta_0(v;\cdot)$ is weakly continuous and such that
    \[
    \int_{\Rn\times \Rn} \eta(x,y) \d\Theta_0(v;(x,y))=\int_{\Rn} \eta(x,\nabla v(x)) \det(\Hess v(x))\d x
    \]
    for $v\in \fconvf\cap C^2(\Rn)$ and continuous $\eta\colon \Rn\times\Rn\to\R$ with compact support with respect to the first variable (see \cite[Theorem 7.3]{colesanti_ludwig_mussnig_3} and Theorem~\ref{thm:theta} below). Let us point out that the Monge--Amp\`ere measure associated with $v$ can be retrieved as a marginal from $\Theta_0(v;\cdot)$, that is, $\MA(v;B)=\Theta_0(v;B \times\Rn)$ for every Borel set $B\subset \Rn$.
    
    \medskip
    
    We establish the following vector-valued analog of the Klain--Schneider theorem for convex functions, where $C_c(\Rn;\Rn)$ denotes the set of continuous function $\psi\colon\Rn\to\Rn$ with compact support. For the importance of the assumption of vertical translation invariance, we refer to \cite{colesanti_ludwig_mussnig_4} and we remark that this property will allow us to connect the appearing operators with functional intrinsic volumes (see Section~\ref{se:intrinsic_moments_minkowski_vectors}).
	 	
    \begin{theorem}
        \label{thm:class_cov_simple}
        A map $\oz\colon \fconvf\to\Rn$ is a continuous, dually translation covariant, vertically translation invariant, dually simple valuation, if and only if there exist functions $\psi\in C_c(\Rn;\Rn)$ and $\zeta\in C_c(\Rn)$ such that
        \begin{equation}
            \label{eq:class_cov_simple}
            \oz(v)=\int_{\Rn}\psi(x) \d\MA(v;x) +  \int_{\Rn\times \Rn} \zeta(x) y \d \Theta_0(v;(x,y))
        \end{equation}
        for every $v\in\fconvf$. 
    \end{theorem}
    \noindent
    Note that the moment vector on convex bodies can be retrieved from the last term \eqref{eq:class_cov_simple}. Indeed, it is easy to show that
    \begin{equation}
        \label{eq:retrieve_moment_vector}
        \int_{\Rn\times\Rn} \zeta(x)y \d \Theta_0(h_K;(x,y))=\zeta(o)\, m(K)
    \end{equation}
    for every convex body $K\in\Kn$, where $o$ denotes the origin in $\Rn$ (see Lemma~\ref{le:retrieve_moment_vector} below).
    
    \medskip
	
    Mimicking the behavior of $K\mapsto m(K)$ on convex bodies, we will consider vector-valued valuations $\oz$ on $\fconvf$ that are \textit{rotation equivariant}. That is, $\oz(v\circ \vartheta^{-1})=\vartheta \oz(v)$ for every $v\in\fconvf$ and $\vartheta\in\SO(n)$. Similarly, we define \textit{$\On$ equivariance} which will be necessary for lower-dimensional cases and which we also call \textit{reflection equivariance} when $n=1$. Assuming such an additional property in Theorem~\ref{thm:class_cov_simple} leads us to the following two operators: for a function $\alpha\in C_c({[0,\infty)})$ let $\om_\alpha^*\colon \fconvf\to\Rn$ be defined as the unique continuous operator such that
    \begin{equation*}
        \om_\alpha^*(v)=\int_{\Rn} \alpha(|x|)\nabla v(x)\det(\Hess v(x))\d x
    \end{equation*}
    for $v\in\fconvf\cap C^2(\Rn)$, where $|x|=\sqrt{\langle x,x\rangle}$ denotes the Euclidean norm of $x\in\Rn$. The above definition continuously extends to $\fconvf$, without additional $C^2$ assumptions, using the Hessian measure $\Theta_0(v;\cdot)$. In addition, let
    \[
    \ot_{n,\xi}^*(v)=\int_{\Rn} \xi(|x|)x \d\MA(v;x)
    \]
    for $v\in\fconvf$, where $\xi$ is a continuous function with bounded support on $(0,\infty)$, the set of which we denote by $C_b((0,\infty))$, with the additional property that $\lim_{t\to 0^+}\xi(t)t=0$. Observe that this condition on $\xi$ is equivalent to the fact that $x\mapsto \xi(|x|)x$ continuously extends from $\Rn\setminus\{o\}$ to $\Rn$ (see Lemma~\ref{le:cond_xi}).
	
    Our following main result, which is a functional analog of Corollary~\ref{cor:moment_vector}, is obtained as a consequence of Theorem~\ref{thm:class_cov_simple}.
	
    \begin{corollary}
        \label{cor:char_m_alpha_t_alpha_simple}
        For $n\geq 3$, a map $\oz\colon\fconvf\to\Rn$ is a continuous, dually translation covariant, vertically translation invariant, rotation equivariant, dually simple valuation, if and only if there exist functions $\alpha\in C_c({[0,\infty)})$ and $\xi\in C_b((0,\infty))$ with $\lim_{t\to 0^+} \xi(t)t=0$ such that
        \[
        \oz(v)=\ot_{n,\xi}^*(v)+\om_\alpha^*(v)
        \]
        for every $v\in\fconvf$. For $n\leq 2$, the same representation holds if we replace rotation equivariance with $\On$ equivariance.
    \end{corollary}

    Let us stress that the new valuation $\ot_{n,\xi}^*$ is dually epi-translation invariant and does not have a (non-trivial) counterpart in the classical theory on convex bodies. We will discuss this in more detail in Section~\ref{se:intrinsic_moments_minkowski_vectors}, where we will introduce \textit{functional intrinsic moments} and \textit{Minkowski vectors}. Before that, however, it should be mentioned that if we replace translation covariance with translation invariance in Corollary~\ref{cor:char_m_alpha_t_alpha_simple}, then we immediately obtain the following characterization of $\ot_{n,\xi}^*$.
	
    \begin{corollary}
        \label{cor:t_alpha_simple}
        For $n\geq 3$, a map $\oz\colon\fconvf\to\Rn$ is a dually simple, continuous, dually epi-translation invariant, and rotation equivariant valuation, if and only if there exists a function $\xi\in C_b((0,\infty))$ with $\lim_{t\to 0^+}\xi(t)t=0$ such that
        \[
        \oz(v)=\ot_{n,\xi}^*(v)
        \]
        for every $v\in\fconvf$. For $n\leq 2$, the same representation holds if we replace rotation equivariance with $\On$ equivariance.
    \end{corollary}

    \paragraph{Plan of This Article}
    We collect basic results on convex functions, Hessian measures, and valuations in Section~\ref{se:preliminaries}. In Section~\ref{se:class_simple} we prove Theorem~\ref{thm:class_cov_simple} and Corollary~\ref{cor:char_m_alpha_t_alpha_simple}. There, we will work in the dual setting of super-coercive, convex functions. After this, we will consider a Steiner-type formula and introduce functional intrinsic moments and Minkowski vectors in Section~\ref{se:intrinsic_moments_minkowski_vectors}. Finally, we present classification results under homogeneity assumptions in Section~\ref{se:class_hom}.
	
    \section{Preliminaries}
    \label{se:preliminaries}
    Let us collect some basic results on convex functions, where we refer to \cite{rockafellar_wets,schneider_cb} as standard references. For a function $w\colon\Rn\to [-\infty,\infty]$, the \textit{convex conjugate} or \textit{Legendre--Fenchel transform} $w^*$ is defined as
    \[
    w^*(y)=\sup\nolimits_{x\in\Rn} \big(\langle x,y \rangle - w(x)\big)
    \]
    for $y\in\Rn$. The image of $\fconvf$ under convex conjugation is
    \[
    \fconvs = \left\{u\colon \Rn \to (-\infty,\infty] : u \text{ is l.s.c., convex, } u\not\equiv \infty\text{, } \lim\nolimits_{|x|\to\infty} \tfrac{u(x)}{|x|}=\infty\right\},
    \]
    the space of proper, lower semicontinuous, super-coercive, convex functions. It is straightforward to check that $u\in\fconvs$, if and only if $u^*\in\fconvf$. Conversely, $v\in\fconvf$, if and only if $v^*\in\fconvs$. Furthermore, convex conjugation is a continuous involution between $\fconvf$ and $\fconvs$, where we equip spaces of convex functions with the topology associated with epi-convergence (see, for example, \cite[Proposition 7.2]{rockafellar_wets} for a description). On $\fconvf$, this is equivalent to pointwise convergence of functions, and on $\fconvs$, this is essentially equivalent to the Hausdorff convergence of level sets \cite[Lemma 5]{colesanti_ludwig_mussnig_1}.
	
    Let us also mention that the composition of convex conjugation with building the Monge--Amp\`ere measure of a function takes the easy form
    \[
    \int_{\Rn} \zeta(x)\d\MA(u^*;x)=\int_{\dom(u)} \zeta(\nabla u(x)) \d x
    \]
    for $u\in\fconvs$ and $\zeta\in C_c(\Rn)$. Here, $\dom(u)=\{x\in\Rn : u(x) <\infty\}$ is the \textit{domain} of $u$ and, by convexity, $u$ is differentiable almost everywhere on the interior of this set. The above is a special case of the more general relation
    \begin{equation}
    \label{eq:theta_conj}
    \int_{\Rn} \zeta(x) y \d\Theta_0(u^*;(x,y))= \int_{\dom(u)} \zeta(\nabla u(x))x \d x
    \end{equation}
    for $u\in\fconvs$ and $\zeta\in C_c(\Rn)$. See, for example, \cite[Section 10.4]{colesanti_ludwig_mussnig_3}.
    
    \medskip
	
    We will use the following result, which is due to \cite[Theorem 17]{colesanti_ludwig_mussnig_4}.
    \begin{theorem}
        \label{thm:theta}
        If $\eta\in C(\Rn\times\Rn)$ has compact support with respect to the first variable, then
        \[
        \oZ(v)=\int_{\Rn\times\Rn} \eta(x,y)\d\Theta_0(v;(x,y))
        \]
        is well-defined for every $v\in\fconvf$ and defines a continuous valuation on $\fconvf$. Moreover,
        \[
        \oZ(v)=\int_{\Rn} \eta(x,\nabla v(x)) \det(\Hess v(x))\d x
        \]
        for every $v\in\fconvf\cap C^2(\Rn)$.
    \end{theorem}
	
    Convex bodies are naturally embedded into $\fconvf$ by associating with each convex body $K\in\Kn$ its \textit{support function}
    \[
    h_K(x)=\max\nolimits_{y\in K} \langle x,y\rangle
    \]
    for $x\in\Rn$. Taking the convex conjugate leads to
    \begin{equation}
    \label{eq:ind_K}
    h_K^*(x)=\ind_K(x)=\begin{cases}
        0\quad &\text{if } x\in K,\\
        \infty\quad &\text{if } x\not\in K,
    \end{cases}
    \end{equation}
    the \textit{(convex) indicator function} of $K$, which is an element of $\fconvs$. We close this section with the following easy result.
	
    \begin{lemma}
        \label{le:retrieve_moment_vector}
        Let $\zeta\in C_c(\Rn)$. For every $K\in\Kn$ we have
        \[
        \int_{\Rn\times\Rn} \zeta(x)y \d \Theta_0(h_K;(x,y))=\zeta(o) m(K).
        \]
        In particular,
        \[
        \om_\alpha^*(h_K)=\alpha(0) m(K)
        \]
        for every $\alpha\in C_c({[0,\infty)})$ and $K\in\Kn$.
    \end{lemma}
    \begin{proof}
        For $K\in\Kn$ it follows from Theorem~\ref{thm:theta}, \eqref{eq:theta_conj}, and \eqref{eq:ind_K} that
        \[
        \int_{\Rn\times\Rn} \zeta(x)y \d \Theta_0(h_K;(x,y)) = \int_{\dom(\ind_K)} \zeta(\nabla \ind_K(x)) x \d x = \zeta(o) \int_K x \d x = \zeta(o)\, m(K).
        \]
        The second part of the statement immediately follows by choosing $\zeta(x)=\alpha(|x|)$.
    \end{proof}

    \section{Classification of Simple Valuations}
    \label{se:class_simple}
    The purpose of this section is to prove Theorem~\ref{thm:class_cov_simple} and Corollary~\ref{cor:char_m_alpha_t_alpha_simple}. We will establish the equivalent dual version of these results on $\fconvs$, Theorem~\ref{thm:class_cov_simple_fconvs}. Let us first explain this equivalence: a map $\oZ\colon\fconvs\to\langle A,+\rangle$ is a valuation, if and only if $v\mapsto \oZ^*(v)=\oZ(v^*)$ is a valuation on $\fconvf$. The operator $\oZ$ is called \textit{translation invariant}, if and only if $\oZ^*$ is dually translation invariant. By the properties of the Legendre--Fenchel transform, this is equivalent to
    \[
    \oZ(u\circ \tau^{-1})=\oZ(u)
    \]
    for every $u\in\fconvs$ and translation $\tau$ on $\Rn$. If $\oZ$ is in addition vertically translation invariant, then we say that $\oZ$ is \textit{epi-translation invariant}. This means that $u\mapsto \oZ(u)$ is invariant under translations of the \textit{epi-graph} of $u$,
    \[
    \epi(u)=\{(x,t)\in \Rn\times\R :  u(x)\leq t\},
    \]
    in $\R^{n+1}$. Moreover, we say that $\oZ$ is \textit{simple} if it vanishes on functions $u\in\fconvs$ whose domain is of dimension less than $n$.
	
    A vector-valued valuation $\oz$ on $\fconvs$ is \textit{translation covariant} if there exists an associated map $\oz^0\colon\fconvs\to\R$ such that
    \[
    \oz(u\circ \tau_x^{-1})=\oz(u)+\oz^0(u)x
    \]
    for every $u\in\fconvs$ and $x\in\Rn$, where $\tau_x$ denotes the translation $y\mapsto y+x$ on $\Rn$. Equivalently, $\oz^*$ is dually translation covariant. Lastly, the definitions of rotation equivariance, $\On$ equivariance, and reflection invariance are the same as for a valuation on $\fconvf$. 
	
    \medskip
	
    We start by establishing the properties of the translation covariant valuations, that we will characterize in the course of this section.
	
    \begin{proposition}
        \label{prop:top_degree_val_fconvs}
        If $\zeta\in C_c(\Rn)$, then
        \begin{equation}
            \label{eq:top_degree_val_fconvs}
            u\mapsto \int_{\dom(u)} \zeta(\nabla u(x))x \d x
        \end{equation}
        is a continuous, translation covariant, vertically translation invariant, and simple valuation on\linebreak$\fconvs$.
    \end{proposition}
    \begin{proof}
        Let $\zeta\in C_c(\Rn)$ be given. We use \eqref{eq:theta_conj} together with Theorem~\ref{thm:theta} (applied to each coordinate of the vector-valued integral above) to see that \eqref{eq:top_degree_val_fconvs} defines a continuous vector-valued valuation $\oz$ on $\fconvs$. For $y\in\Rn$ and $c\in\R$ we have
        \begin{align*}
            \oz(u\circ \tau_{y}^{-1} + c)&= \int_{\dom(u\circ \tau_{y}^{-1}+c)} \zeta(\nabla (u\circ \tau_{y}^{-1}+c) (x)) x \d x\\
            &=\int_{\dom(u\circ \tau_{y}^{-1})} \zeta(\nabla u(x-y)) x \d x\\
            &=\int_{\dom(u)} \zeta(\nabla u(x))(x+y) \d x\\
            &=\oz(u)+y \int_{\dom(u)} \zeta(\nabla u(x))\d x
        \end{align*}
        for every $u\in\fconvs$, which shows that $\oz$ is translation covariant and vertically translation invariant. Lastly, if $u$ has lower-dimensional domain, then the use of the $n$-dimensional Lebesgue measure in \eqref{eq:top_degree_val_fconvs} shows that $\oz(u)=o$ or, equivalently, $\oz$ is simple.
    \end{proof}
	
    Next, we consider the map associated with a translation covariant valuation. The following result is a functional analog of a well-known fact for translation covariant valuations on convex bodies. See, for example, \cite[Theorem 10.5]{mcmullen_schneider}.
    \begin{lemma}
        \label{le:properties_oz_0}
        If $\oz\colon\fconvs\to\Rn$ is a continuous, translation covariant, vertically translation invariant valuation, then the associated map $\oz^0\colon\fconvs\to\R$ is a continuous, epi-translation invariant valuation. In addition, if $\oz$ is simple, then so is $\oz^0$.
    \end{lemma}
    \begin{proof}
        Let a continuous, translation covariant, vertically translation invariant valuation $\oz$ on $\fconvs$ be given that takes values in $\Rn$. We first show that the associated map $\oz^0$ is a valuation. For this let $u,w\in\fconvs$ be such that $u\vee w, u\wedge w\in\fconvs$. Observe that
        \[
        (u\circ \tau_y^{-1})\vee (w\circ \tau_y^{-1}) = (u\vee w)\circ \tau_y^{-1}\quad \text{and} \quad (u\circ \tau_y^{-1})\wedge (w\circ \tau_y^{-1}) = (u\wedge w)\circ \tau_y^{-1} 
        \]
        for $y\in\Rn$. Thus, since $\oz$ is a translation covariant valuation,
        \begin{align*}
            \oz(u) + \oz(w) + (\oz^0(u)+\oz^0(w)) y&= \oz(u\circ \tau_y^{-1})+ \oz(w\circ \tau_y^{-1})\\
            &=\oz((u\circ \tau_y^{-1}) \vee (w \circ \tau_y^{-1})) + \oz((u\circ \tau_y^{-1}) \wedge (w \circ \tau_y^{-1}))\\
            &= \oz(u \vee w) + \oz(u\wedge w) + (\oz^0(u \vee w) + \oz^0(u\wedge w))y\\
            &= \oz(u) +\oz(w) + (\oz^0(u\vee w) + \oz^0 (u\wedge w))y
        \end{align*}
        for $y\in\Rn$, which shows the desired property
        \[
        \oz^0(u)+\oz^0(w) = \oz^0(u\vee w) + \oz^0 (u\wedge w).
        \]
        Next, for $u\in\fconvs$, $x,y\in\Rn$, and $c\in\R$, it follows from the translation covariance and vertical translation invariance of $\oz$ that
        \begin{align*}
            \oz(u)+\oz^0(u)(x+y) &= \oz(u\circ \tau_{x+y}^{-1})\\
            &= \oz((u+c)\circ \tau_x^{-1}\circ \tau_y^{-1})\\
            &= \oz((u+c)\circ \tau_x^{-1}) + \oz^0((u+c)\circ\tau_x^{-1})y\\
            &=\oz(u)+\oz^0(u)x + \oz^0((u+c)\circ\tau_x^{-1})y.
        \end{align*}
        Hence,
        \[
        \oz^0(u\circ \tau_x^{-1}+c) = \oz^0(u),
        \]
        which means that $\oz^0$ is epi-translation invariant. Similarly, if $u_k$, $k\in\N$, is an epi-convergent sequence in $\fconvs$ with limit $u\in\fconvs$, then for every $y\in\Rn$, $u_k\circ \tau_y^{-1}$ epi-converges to $u\circ \tau_y^{-1}$ as $k\to\infty$. Thus, by the continuity of $\oz$,
        \begin{align*}
            \oz(u)+\lim\nolimits_{k\to\infty} \oz^0(u_k)y &=\lim\nolimits_{k\to\infty}\big(\oz(u_k)+\oz^0(u_k)y\big)\\
            &=\lim\nolimits_{k\to\infty} \oz(u_k\circ \tau_y^{-1})\\
            &=\oz(u\circ \tau_y^{-1})\\
            &=\oz(u)+\oz^0(u)y,
        \end{align*}
        which shows that $\oz^0$ is continuous. Lastly, if $u\in\fconvs$ is such that the dimension of $\dom(u)$ is smaller than $n$, then so is the dimension of $\dom(u\circ \tau_y^{-1})$ for every $y\in\Rn$. Thus, if $\oz$ is simple, we have
		\[
		o=\oz(u\circ \tau_y^{-1})=\oz(u)+\oz^0(u)y = o + \oz^0(u)y
		\]
		for every $u\in\fconvs$ and $y\in\Rn$, and therefore also $\oz^0$ has to be simple.
    \end{proof}
	
    Lemma~\ref{le:properties_oz_0} indicates that studying continuous, translation covariant, vertically translation invariant, simple valuations goes hand in hand with investigating continuous, epi-translation invariant, simple, real-valued valuations on $\fconvs$. The latter is the content of the Klain--Schneider theorem on $\fconvs$, which was proven in \cite[Theorem 1.2]{colesanti_ludwig_mussnig_8} and which is dual to Theorem~\ref{thm:klain_schneider_fconvf}.
    \begin{theorem}
        \label{thm:char_simple_real_val}
        A map $\oZ\colon \fconvs\to\R$ is a continuous, epi-translation invariant, simple valuation, if and only if there exists a function $\zeta\in C_c(\Rn)$ such that
        \[
        \oZ(u)=\int_{\dom(u)} \zeta(\nabla u(x))\d x
        \]
        for every $u\in\fconvs$.
    \end{theorem}
	
    We furthermore need the following two results to deal with rotation equivariance.
	
    \begin{lemma}
        \label{le:cond_xi}
        Let $\xi\colon (0,\infty)\to\R$ and $\psi\colon\Rn\to\Rn$ be such that
        \[
        \psi(x)=\xi(|x|)x
        \]
        for $x\in\Rn\setminus\{o\}$. The map $\psi$ can be chosen from $C_c(\Rn;\Rn)$, if and only if $\xi$ is continuous with bounded support on $(0,\infty)$ such that
        \begin{equation}
            \label{eq:cond_xi}
            \lim\nolimits_{t\to 0^+} \xi(t)t =0.
        \end{equation}
    \end{lemma}
    \begin{proof}
        Suppose first that $\psi$ is continuous with compact support on $\Rn$. Denoting the first coordinate of the vector-valued function $\psi$ by $\psi_1$, we see that
        \[
        \xi(t)=\xi(|t e_1|) = \frac{\psi_1(t e_1)}{t}
        \]
        for $t>0$, where $e_1$ is the first vector of the standard orthonormal basis of $\Rn$ . Thus, $\xi$ is continuous with bounded support on $(0,\infty)$. Furthermore, since $\psi$ is continuous in $o$,
        \[
        \lim\nolimits_{t\to 0^+} \xi(|t|)t e_1 =\lim\nolimits_{t\to 0^+} \psi(t e_1)=\psi(o)=\lim\nolimits_{t\to 0^+} \psi(-t e_1)= \lim\nolimits_{t\to 0^+} \xi(|t|)(-t e_1),
        \]
        which implies \eqref{eq:cond_xi}.
        
        Assume now that $\xi$ is continuous with bounded support on $(0,\infty)$ and satisfies \eqref{eq:cond_xi}. Clearly, $\psi$ is continuous with bounded support on $\Rn\setminus\{o\}$. In addition, for any sequence  $x_i$, $i\in\N$, in $\Rn$ such that $\lim_{i\to\infty} x_i=o$, we have
        \[
        \lim\nolimits_{i\to\infty} |\psi(x_i)| = \lim\nolimits_{i\to\infty} \big|\xi(|x_i|) x_i\big| = 0,
        \]
        which shows that $\psi$ continuously extends to $o$ with $\psi(o)=0$.
    \end{proof}
	
    \begin{lemma}
        \label{le:g_commutes}
        Let $\psi\in C_c(\Rn;\Rn)$. For $n\geq 3$, the map $\psi$ satisfies
        \begin{equation}
            \label{eq:g_commutes}
            \psi(\vartheta y)=\vartheta \psi(y)
        \end{equation}
        for every $\vartheta\in\SOn$ and $y\in\Rn\setminus\{o\}$, if and only if there exists $\xi\in C_b((0,\infty))$ with $\lim_{t\to 0^+} \xi(t)t=0$ such that
        \[
        \psi(y)=\xi(|y|)y
        \]
        for every $y\in\Rn\setminus\{o\}$. For $n\leq 2$, the same representation holds if $\psi$ is assumed to commute with $\On$ instead.
    \end{lemma}
    \begin{proof}
        If $\xi\in C_b((0,\infty))$ is such that $\lim_{t\to 0^+} \xi(t)t=0$, then it easily follows from Lemma~\ref{le:cond_xi} that $\psi(y)=\xi(|y|)y$, $y\in\Rn\setminus\{o\}$, has the desired properties.
        
        Now for $n\geq 3$, let a continuous vector-valued map $\psi$ on $\Rn$ be given that commutes with $\SOn$. If for $y\in\Rn\setminus\{o\}$ we denote
        \[
        S_y=\{\vartheta\in\SOn : \vartheta y = y\},
        \]
        then \eqref{eq:g_commutes} shows that $\vartheta \in S_y$ if and only if $\vartheta \in S_{\psi(y)}$. In particular, $\psi(y)$ and $y$ are parallel, which means that there exists $\hat{\xi}\colon (0,\infty) \times \sn\to\R$ such that
        \[
        \psi(r z)=\hat{\xi}(r,z) r z
        \]
        for $r\in(0,\infty)$ and $z\in\sn$. By \eqref{eq:g_commutes} we have
        \[
        \hat{\xi}(r,\vartheta z)r \vartheta z = \psi(r\vartheta z)=\vartheta \psi(rz)=\hat{\xi}(r,z)r\vartheta z
        \]
        for $\vartheta\in\SOn$, and thus $\hat{\xi}(r,z)$ is independent of $z\in\sn$. This means that there exists $\xi\colon (0,\infty)\to\R$ such that
        \[
        \hat{\xi}(r,z)=\xi(r)
        \]
        for $r\in(0,\infty)$ and $z\in\sn$ and, therefore, $\psi(y)=\xi(|y|)y$ for $y\in\Rn\setminus\{o\}$. The conclusion now follows from Lemma~\ref{le:cond_xi}.
        
        Next, in case $n=2$, let $\psi$ on $\R^2$ be such that $\psi(\vartheta y)=\vartheta \psi(y)$ for $\vartheta\in \Ot$ and $y\in\R^2$. Writing $\rho_y\in\Ot$ for the reflection at the line spanned by $y\in\R^2\setminus\{o\}$, we have
        \[
        \psi(y)=\psi(\rho_y y)=\rho_y \psi(y).
        \]
        This shows that $\psi(y)$ must be parallel to $y$ we can therefore use the same argumentation as in the case $n\geq 3$.
        
        Lastly, for $n=1$, we assume that the map $\psi$ satisfies $\psi(-y)=-\psi(y)$ for $y\in\R$. This is equivalent to $\psi(0)=0$ and
        \[
        \psi(y)=\frac{\psi(|y|)}{|y|}y
        \]
        for $y\neq 0$. The statement now follows with $\xi(t)=\psi(t)/t$ for $t\in(0,\infty)$.	 
    \end{proof}

    We can now prove the main result of this section, which is equivalent to Theorem~\ref{thm:class_cov_simple} and Corollary~\ref{cor:char_m_alpha_t_alpha_simple}. For $\alpha\in C_c({[0,\infty)})$ we consider the operator $\om_{\alpha}\colon\fconvs\to\Rn$, which is defined as $\om_{\alpha}(u)=\om_{\alpha}^*(u^*)$, or equivalently, by \eqref{eq:theta_conj},
    \[
    \om_{\alpha}(u)=\int_{\dom(u)} \alpha(|\nabla u(x)|) x \d x
    \]
    for $u\in\fconvs$. Furthermore, for $\xi\in C_b((0,\infty))$ such that $\lim_{t\to 0^+} \xi(t)t=0$, we define $\ot_{n,\xi}\colon\fconvs\to\Rn$ as $\ot_{n,\xi}(u)=\ot_{n,\xi}^*(u^*)$, which is equivalent to
    \[
    \ot_{n,\xi}(u)=\int_{\dom(u)} \xi(|\nabla u(x)|) \nabla u(x) \d x
    \]
    for $u\in\fconvs$.
	
    \begin{theorem}
        \label{thm:class_cov_simple_fconvs}
        A map $\oz\colon \fconvs\to\Rn$ is a continuous, translation covariant, vertically translation invariant, simple valuation, if and only if there exist functions $\psi\in C_c(\Rn;\Rn)$ and $\zeta\in C_c(\Rn)$ such that
        \begin{equation}
            \label{eq:class_cov_simple_fconvs}
            \oz(u)=\int_{\dom(u)}\psi(\nabla u(x)) \d x +  \int_{\dom(u)} \zeta(\nabla u(x)) x \d x
        \end{equation}
        for every $u\in\fconvs$. For $n\geq 3$, the valuation $\oz$ is in addition rotation equivariant, if and only if there exist functions $\alpha\in C_c({[0,\infty)})$ and $\xi\in C_b((0,\infty))$ with $\lim_{t\to 0^+} \xi(t)t=0$ such that
        \begin{equation}
            \label{eq:class_cov_equiv_simple_fconvs}
            \oz(u)=\ot_{n,\xi}(u)+\om_\alpha(u)
        \end{equation}
        for every $u\in\fconvf$. For $n\leq 2$, the same representation holds if we replace rotation equivariance with $\On$ equivariance.
    \end{theorem}

    \begin{remark}
        If in Theorem~\ref{thm:class_cov_simple_fconvs} translation covariance and vertical translation invariance are replaced by epi-translation invariance, then only the valuations
        \[
        u\mapsto \int_{\dom(u)}\psi(\nabla u(x)) \d x
        \]
        and $\ot_{n,\xi}$, respectively, remain. Furthermore, let us point out that for $n\geq 2$ the representation \eqref{eq:class_cov_equiv_simple_fconvs} does not hold if one merely assumes $\SO(n)$ equivariance. For $n=1$, we have $\SO(1)=\{\operatorname{Id}\}$ and thus, also \eqref{eq:class_cov_simple_fconvs} is $\SO(1)$ equivariant. For $n=2$, let $ \Phi\colon [0,\infty)\to \SO(2)$ be continuous and let $\xi\in C_b((0,\infty))$ be such that $\lim_{t\to 0^+} \xi(t)t=0$. Since $\SO(2)$ is commutative, also the valuation
        \begin{equation}
        \label{eq:so2_equiv}
        u\mapsto \int_{\dom(u)} \xi(|\nabla u(x)|) \Phi(|\nabla u(x)|)\nabla u(x) \d x
        \end{equation}
        is $\SO(2)$ equivariant and generalizes $\ot_{n,\xi}$. In fact, one can show that every continuous, translation covariant, vertically translation invariant, $\SO(2)$ equivariant, simple valuation can be written as the sum of \eqref{eq:so2_equiv} and $\om_{\alpha}$ with some $\alpha\in C_c({[0,\infty)})$.
    \end{remark}
    \begin{proof}[Proof of Theorem~\ref{thm:class_cov_simple_fconvs}]
        For $\psi\in C_c(\Rn;\Rn)$ and $\zeta\in C_c(\Rn)$ it follows from Theorem~\ref{thm:char_simple_real_val}, applied coordinate-wise, together with Proposition~\ref{prop:top_degree_val_fconvs}
        that \eqref{eq:class_cov_simple_fconvs} defines a continuous, translation covariant, vertically translation invariant, simple valuation.
        
        Now let a continuous, translation covariant, vertically translation invariant, simple valuation $\oz\colon\fconvs\to\Rn$ be given. By Lemma~\ref{le:properties_oz_0} together with Theorem~\ref{thm:char_simple_real_val}, there exists $\zeta\in C_c(\Rn)$ such that the associated map $\oz^0$ is of the form
        \[
        \oz^0(u)=\int_{\dom(u)} \zeta(\nabla u(x))\d x
        \]
        for $u\in\fconvs$. We now define $\tilde{\oz}\colon\fconvs\to\Rn$ as
        \[
        \tilde{\oz}(u)=\oz(u)-\int_{\dom(u)} \zeta(\nabla u(x))x \d x
        \]
        for $u\in\fconvs$. For any $y\in\Rn$ we have
        \begin{align*}
            \tilde{\oz}(u\circ &\tau_{y}^{-1})\\
            &= \oz(u)+\left(\int_{\dom(u)} \zeta(\nabla u(x))\d x\right) y-\int_{\dom(u)} \zeta(\nabla u(x))x \d x - \left(\int_{\dom(u)} \zeta(\nabla u(x))\d x\right) y\\
            &=\tilde{\oz}(u),
        \end{align*}
        which together with Proposition~\ref{prop:top_degree_val_fconvs} and the properties of $\oz$ shows that $\tilde{\oz}$ is a continuous, epi-translation invariant, simple valuation on $\fconvs$. Thus, another application of Theorem~\ref{thm:char_simple_real_val}, this time applied to each of the coordinates of the vector-valued map $\tilde{\oz}$, shows there exist functions $\psi_1,\ldots,\psi_n\in C_c(\Rn)$ such that the $i$th coordinate of $\tilde{\oz}$, $i\in\{1,\ldots,n\}$, must be of the form
        \[
        u\mapsto \int_{\dom(u)} \psi_i(\nabla u(x)) \d x,
        \]
        concluding the proof of the first part of the statement.
        
        Next, let $\alpha\in C_c({[0,\infty)})$ and $\xi\in C_b((0,\infty))$ with $\lim_{t\to 0^+} \xi(t)t=0$ be given, and let $\oz\colon\fconvs\to\Rn$ be as in \eqref{eq:class_cov_equiv_simple_fconvs}. We will show that $\oz$ is $\On$ equivariant. If $u\in\fconvs$ and $\vartheta\in\On$, then
        \begin{align*}
            \ot_{n,\xi}(u\circ \vartheta^{-1})&=\int_{\dom(u\circ \vartheta^{-1})} \xi(|\nabla (u\circ \vartheta^{-1})(x)|) \nabla (u\circ \vartheta^{-1})(x) \d x\\
            &= \vartheta \int_{\dom(u\circ \vartheta^{-1})} \xi(|\vartheta \nabla u(\vartheta^{-1}x)|) \nabla u(\vartheta^{-1}x)\d x\\
            &= \vartheta \ot_{n,\xi}(u).
        \end{align*}
        Similarly, 
        \begin{align*}
            \om_{\alpha}(u\circ \vartheta^{-1}) &= \int_{\dom(u\circ \vartheta^{-1})} \alpha(|\nabla(u\circ \vartheta^{-1})(x)|)x \d x\\
            &= \int_{\dom(u\circ \vartheta^{-1})} \alpha(|\vartheta \nabla u (\vartheta^{-1}x)|) x \d x\\
            &= \vartheta \int_{\dom(u)} \alpha(|\nabla u(x)|) \d x\\
            &= \vartheta \om_{\alpha}(u),
        \end{align*}
        which shows the desired property.
        
        Conversely, assume that $\oz$ is a continuous, translation covariant, vertically translation invariant, rotation equivariant, simple valuation on $\fconvs$. By the first part of the proof there exist functions $\psi\in C_c(\Rn;\Rn)$ and $\zeta\in C_c(\Rn)$ such that
        \[
        \oz(u)=\int_{\dom(u)} \psi(\nabla u(x)) \d x + \int_{\dom(u)} \zeta(\nabla u(x))x \d x
        \]
        for every $u\in\fconvs$. By the rotation equivariance of $\oz$ we see that
        \begin{align*}
            o&=\oz(u\circ \vartheta^{-1})- \vartheta \oz(u)\\
            &=\int_{\dom(u\circ \vartheta^{-1})} \psi(\vartheta \nabla u(\vartheta^{-1}x)) \d x + \int_{\dom(u\circ \vartheta^{-1})} \zeta(\vartheta \nabla u(\vartheta^{-1}x)) x \d x\\
            &\quad - \vartheta \int_{\dom(u)} \psi(\nabla u(x)) \d x - \vartheta \int_{\dom(u)} \zeta(\nabla u(x)) x \d x\\
            &= \int_{\dom(u)} \psi(\vartheta \nabla u(x)) - \vartheta \psi(\nabla u(x)) \d x + \vartheta \int_{\dom(u)} \big(\zeta(\vartheta \nabla u(x))-\zeta(\nabla u(x)) \big)x \d x
        \end{align*}
        for $u\in\fconvs$ and $\vartheta\in \SOn$. Choosing $u=\ind_K+\langle y,\cdot\rangle$ for some $K\in\Kn$ and $y\in\Rn$ gives
        \[
        o=\vol_n(K)\big(\psi(\vartheta y)-\vartheta \psi(y)\big) + (\zeta(\vartheta y)-\zeta(y))\vartheta\, m(K).
        \]
        Thus, considering first a convex body $K$ with $\vol_n(K)>0$ and $m(K)=o$, and then a body $K$ with $m(K)\neq 0$, shows that
        \[
        \psi(\vartheta y) = \vartheta \psi(y)  \quad \text{and}\quad \zeta(\vartheta y)=\zeta(y)
        \]
        for every $\vartheta\in\SO(n)$ (or $\vartheta\in\On$ for $n\leq 2$) and $y\in\Rn$. By Lemma~\ref{le:g_commutes} this implies that there exists $\xi\in C_b((0,\infty))$ with $\lim_{t\to 0^+} \xi(t)t=0$ such that
        \[
        \psi(y)=\xi(|y|)y
        \]
        for every $y\in\Rn\setminus\{o\}$. Furthermore, it easily follows that $\zeta(y)=\alpha(|y|)$ for some $\alpha\in C_c({[0,\infty)})$, where we choose $\vartheta(y)=-y$ in case $n=1$.
    \end{proof}
    
    \section{Functional Intrinsic Moments and Minkowski Vectors}
    \label{se:intrinsic_moments_minkowski_vectors}
    Let $\alpha\in C_c({[0,\infty)})$ and let $q(x)=|x|^2/2$. Considering a Steiner-type formula, starting with the operator $\om_\alpha^*$, results in
    \begin{equation}
        \label{eq:om_alpha_steiner}
        \om_\alpha^*(v+r q)= \sum_{j=0}^n r^{n-j} \int_{\Rn} \alpha(|x|)\nabla v(x) [\Hess v(x)]_j \d x + r^{n-j+1} \int_{\Rn} \alpha(|x|)x [\Hess v(x)]_j \d x
    \end{equation}
    for $v\in\fconvf\cap C^2(\Rn)$ and $r\geq 0$. Here, we write $[A]_j$ for the \textit{$j$th elementary symmetric function} of the eigenvalues of a symmetric matrix $A\in\R^{n\times n}$, with the convention $[A]_0\equiv 1$. In particular, $[\Hess v(x)]_n=\det(\Hess v(x))$. Using appropriate Hessian measures, the above continuously extends to general $v\in\fconvf$ (see \cite[Theorem 17]{colesanti_ludwig_mussnig_4}).
	
    \medskip
	
    Equation \eqref{eq:om_alpha_steiner} gives rise to two families of operators. For $\alpha\in C_c({[0,\infty)})$ and $j\in\{0,\ldots,n\}$, we define $\oz_{j+1,\alpha}^*\colon\fconvf\to\Rn$ to be the unique continuous map such that
    \[
    \oz_{j+1,\alpha}^*(v)=\int_{\Rn} \alpha(|x|)\nabla v(x) [\Hess v(x)]_j \d x
    \]
    for $v\in\fconvf\cap C^2(\Rn)$, and we call operators of this form \textit{functional intrinsic moments}. This is motivated by the observation that when $v$ is the support function $h_K$ of a convex body $K$, then we retrieve the \textit{intrinsic moment} $z_{j+1}(K)$, which in turn is obtained as a coefficient in a Steiner formula, starting with the moment vector $m(K)$ (we refer to \cite[Section 5.4.1]{schneider_cb} for details). For the case $j=n$, this follows from \eqref{eq:retrieve_moment_vector} together with the relation $z_{n+1}(K)=m(K)$. When $j\in\{0,\ldots,n-1\}$ and $h_K$ is of class $C^2$, then
    \begin{align*}
        \oz_{j+1,\alpha}^*(h_K) &\simeq \int_{\sn} \nabla h_K(z)\,[\Hess h_K(z)]_j \d\hm^{n-j}(z)\\
        &\simeq \int_{\sn} \nabla h_K(z) \d S_j(K,z)\\
        &\simeq  \int_{\Rn} x \d C_j(K,x)\\
        &\simeq z_{j+1}(K),
    \end{align*}
    where $\simeq$ denotes equality up to constant factors, depending on $n$, $j$, and $\alpha$. Here, we write $\hm^k$ for the $k$-dimensional \textit{Hausdorff measure}, $S_j(K,\cdot)$ denotes the $j$th \textit{area measure} of $K$, and $C_j(K,\cdot)$ is its $j$th \textit{curvature measure} (see \cite[p.\ 214]{schneider_cb} for their definitions). By continuity we now retrieve $\oz_{j+1,\alpha}^*(h_K)\simeq z_{j+1}(K)$ for general $K\in\Kn$.
    
    Similar to their counterparts on convex bodies, the behavior of functional intrinsic moments under translations is described by functional intrinsic volumes. Indeed,
    \begin{equation}
        \label{eq:oz_j_alpha_transl_cov}
        \oz_{j+1,\alpha}^*(v+\langle y,\cdot \rangle) = \int_{\Rn} \alpha(|x|) \nabla v(x)[\Hess v(x)]_j \d x + y \int_{\Rn} \alpha(|x|) [\Hess v(x)]_j \d x 
    \end{equation}
    for $v\in \fconvf\cap C^2(\Rn)$ and $y\in\Rn$, where
    \[
    \oZZ{j}{\alpha}^*(v)=\int_{\Rn} \alpha(|x|) [\Hess v(x)]_j \d x
    \]
    is one of the functional intrinsic volumes that were introduced and characterized in \cite{colesanti_ludwig_mussnig_5}.
    
    \medskip
    
    The second family of operators coming from \eqref{eq:om_alpha_steiner} is of the form
    \[
    \ot_{j,\alpha}^*(v)=\int_{\Rn} \alpha(|x|)x\, [\Hess v(x)]_j \d x
    \]
    for $v\in\fconvf\cap C^2(\Rn)$, $j\in\{0,\ldots,n\}$, and $\alpha\in C_c({[0,\infty)})$. Let now $v=h_K$ with $K\in\Kn$. In case $j=n$, it follows from \eqref{eq:theta_conj} and \eqref{eq:ind_K} that
    \[
    \ot_{n,\alpha}^*(h_K)=\int_{\Rn} \alpha(|x|) x \d\MA(h_K;x) =\int_{\dom(\ind_K)} \alpha(|\nabla \ind_K(x)|) \nabla \ind_K(x) \d x = \vol_n(K) \alpha(0)\,o= o, 
    \]
    where we recall that $o\in\Rn$ denotes the origin. Furthermore, for $j\in\{0,\ldots,n-1\}$, we obtain
    \begin{equation}
        \label{eq:ot_j_alpha_h_K}
        \ot_{j,\alpha}^*(h_K)\simeq \int_{\sn} z \d S_j(K,z) = o,
    \end{equation}
    where the last equality is a general form of the Minkowski relation (see, for example, \cite[(5.30) and (5.52)]{schneider_cb}). Since \eqref{eq:ot_j_alpha_h_K} is a \textit{Minkowski tensor} of rank $1$ (the only non-vanishing Minkowski tensors of rank $1$ are the intrinsic moments; cf.\ \cite[Section 5.4.2]{schneider_cb}), we will refer to the valuations $\ot_{j,\alpha}^*$ as \textit{functional Minkowski vectors}. These operators vanish not only on support functions but also on every radially symmetric convex function, which is a consequence of the rotation equivariance of $\ot_{j,\alpha}^*$. Of course, this raises the question of whether functions $v\in\fconvf$ exist at all, such that $\ot_{j,\alpha}^*(v)\neq o$. Indeed,
    \[
    \ot_{0,\alpha}^*(v) = \int_{\Rn} \alpha(|x|) x \d x = o
    \]
    for every $v\in\fconvf$, but the remaining operators in this family generally do not vanish. For the case $j=n$, this easily follows from the fact that $\MA(h_K \circ \tau_x^{-1};\cdot)$ is a Dirac measure concentrated at $x\in\Rn$ so that
    \[
    \ot_{n,\alpha}^*(h_K\circ \tau_x^{-1}) = \vol_n(K) \alpha(|x|)x
    \]
    for every $K\in\Kn$. Explicit examples for the remaining cases will be presented in \cite{mouamine_mussnig_2}.
    
    \medskip
    
    We remark that results for functional intrinsic volumes \cite{colesanti_ludwig_mussnig_5} suggest that the operators $\oz_{j+1,\alpha}^*$ and $\ot_{j,\alpha}^*$ may also exist for densities $\alpha$ with a singularity at $0^+$, where the admissible singularities possibly depend on $j$ and $n$. Indeed, we have already seen in Corollary~\ref{cor:char_m_alpha_t_alpha_simple} that $\ot_{n,\xi}^*$ can be defined for $\xi\in C_b((0,\infty))$ such that $\lim_{t\to 0^+} \xi(t)t=0$. A treatment of the possible singularities for the operators $\ot_{j,\alpha}^*$ together with a complete characterization in a Hadwiger-type theorem will be the subject of a subsequent paper \cite{mouamine_mussnig_2}. For the functional intrinsic moments, equation \eqref{eq:oz_j_alpha_transl_cov} indicates that similar singularities may appear as for functional intrinsic volumes. However, the fact that in the case of the latter, singular integrals exist in a Riemann sense and not in a Lebesgue sense poses an additional challenge. See also \cite{knoerr_singular}.
    
    Lastly, let us point out that different functional analogs of moment vectors were previously considered, for example, for $L^p$-spaces \cite{tsang_mink} or log-concave functions \cite{mussnig_lc}.
    
    \section{Classification of Homogeneous Valuations}
    \label{se:class_hom}
    In addition to the main results presented in this article, we establish classification results for homogeneous valuations of extremal degrees. Here, $\oz\colon\fconvf\to\R$ is \textit{homogeneous} of degree $s\in\R$ if $\oz(\lambda v)=\lambda^s \oz(v)$ for every $v\in\fconvf$ and $\lambda >0$. The following decomposition result for homogeneous valuations was shown in \cite[Theorem 4]{colesanti_ludwig_mussnig_4} (and later also in \cite{knoerr_support}).
    
    \begin{theorem}
		\label{thm:mcmullen_fconvf}
		If $\oZ\colon\fconvf\to\R$ is a continuous, dually epi-translation invariant valuation, then there exist continuous, dually epi-translation invariant valuations $\oZ_0,\ldots,\oZ_n\colon\fconvf\to\R$ such that $\oZ_j$ is homogeneous of degree $j$, $j\in\{0,\ldots,n\}$, and $\oZ=\oZ_0+\cdots+\oZ_n$.
    \end{theorem}
    
    Similarly to Theorem~\ref{thm:mcmullen_fconvf}, it follows from recent results on polynomial valuations \cite[Theorem 1.3]{knoerr_ulivelli_polynomial}, that every continuous, dually translation covariant, vertically translation invariant valuation can be written as a sum of homogeneous valuations with degrees ranging from $0$ to $(n+1)$.
	
    \paragraph{Homogeneity of Degree $\boldsymbol{0}$}
    We start with valuations that are homogeneous of degree $0$. For this we will use the following result from \cite[Theorem 26]{colesanti_ludwig_mussnig_4}
    
    \begin{theorem}
        \label{thm:class_zero}
        A map $\oZ\colon\fconvf\to\R$ is a continuous, dually epi-translation invariant valuation that is homogeneous of degree $0$, if and only if $\oZ$ is constant.
    \end{theorem}
    
	
    
    Next, we establish an elementary result, similar to Lemma~\ref{le:properties_oz_0}.
    
    \begin{lemma}
        \label{le:oz_0_hom}
        If $\oz\colon\fconvf\to\Rn$ is a dually translation covariant valuation that is homogeneous of degree $s+1$, then the associated map $\oz^0\colon\fconvf\to\R$ is homogeneous of degree $s$.
    \end{lemma}
    \begin{proof}
        Let $\oz$ be given. We have
        \begin{align*}
            \lambda^{s+1}\oz(v)+\lambda^{s+1}\oz^0(v)x &= \oz(\lambda (v+\langle x,\cdot\rangle ))\\
            &=\oz(\lambda v+\langle \lambda x,\cdot \rangle)\\
            &=\oz(\lambda v) + \oz(\lambda v) \lambda x\\
            &=\lambda^{s+1}\oz(v)+ \lambda \oz(\lambda v) x
        \end{align*}
        for $v\in\fconvf$, $x\in\Rn$, and $\lambda > 0$. Therefore, $\oz^0$ is homogeneous of degree $s$.
    \end{proof}

    It is now straightforward to show the following classification.
    
    \begin{theorem}
        A map $\oz\colon\fconvf\to\Rn$ is a continuous, dually translation covariant, vertically translation invariant valuation that is homogeneous of degree $0$, if and only if there exists $t\in \Rn$ such that $\oz(v)=t$ for every $v\in\fconvf$. For $n\geq 2$, the valuation $\oz$ is in addition rotation equivariant, if and only if it is identically $o$. For $n= 1$, the same representation holds if $\oz$ is reflection equivariant.
    \end{theorem}
    \begin{proof}
        If $\oz$ is as in the statement, then Lemma~\ref{le:properties_oz_0} and Lemma~\ref{le:oz_0_hom} show that the associated map $\oz_0$ must be a continuous, dually epi-translation invariant valuation that is homogeneous of degree $-1$. By Theorem~\ref{thm:mcmullen_fconvf} we must have $\oz^0\equiv 0$, and thus, $\oz$ is dually translation invariant (and therefore also dually epi-translation invariant). The result now follows from applying Theorem~\ref{thm:class_zero} to each coordinate of $\oz$.
    \end{proof}
	
    \paragraph{Homogeneity of Top Degrees}	
    Let us now turn our attention to valuations that are homogeneous of top degrees. As in Section~\ref{se:class_simple}, we will continue to work on the space $\fconvs$ for the remainder of this section. We say that $\oz\colon\fconvs\to\Rn$ is \textit{epi-homogeneous} of degree $s\in\R$ if $\oz(\lambda \sq u)=\lambda^s \oz(u)$ for every $u\in\fconvs$ and $\lambda >0$. Here, $\lambda\sq u$ is the \textit{epi-multiplication} of $u$ with $\lambda$, which is defined as
    \[
    \lambda\sq u(x)=\lambda\, u\left( \frac x\lambda \right)
    \]
    for $x\in\R^n$. We remark that this definition continuously extends to $0\sq u=\ind_{\{o\}}$. This is equivalent to $(\lambda \sq u)^* = \lambda u^*$ and $u\mapsto \oz(u)$ is epi-homogeneous, if and only if $v\mapsto \oz(v^*)$ is homogeneous.
    
    \medskip
    
    We start with epi-translation invariant valuations. We will use the following result from \cite[Theorem 2]{colesanti_ludwig_mussnig_4}.
	
    \begin{theorem}
        \label{thm:char_top_degree_real_val}
        A map $\oZ\colon\fconvs\to\R$ is a continuous, epi-translation invariant valuation that is epi-homogeneous of degree $n$, if and only if there exists a function $\zeta\in C_c(\Rn)$ such that
        \[
        \oZ(u)=\int_{\dom(u)} \zeta(\nabla u(x))\d x
        \]
        for every $u\in\fconvs$.
    \end{theorem}
    
    We can now prove the following classification result for epi-translation invariant valuations that are epi-homogeneous of degree $n$.
    
    \begin{theorem}
        \label{thm:t_alpha_fconvs}
        A map $\oz\colon\fconvs\to\Rn$ is a continuous, epi-translation invariant valuation that is epi-homogeneous of degree $n$, if and only if there exists a function $\psi\in C_c(\Rn;\Rn)$ such that
        \[
        \oz(u)=\int_{\dom(u)} \psi(\nabla u(x))\d x
        \]
        for every $u\in\fconvs$. For $n\geq 3$, the valuation $\oz$ is in addition rotation equivariant, if and only if there exists a function $\xi\in C_b((0,\infty))$ with $\lim_{t\to 0^+} \xi(t)t=0$ such that
        \[
        \oz(u)=\ot_{n,\xi}(u)
        \]
        for every $u\in\fconvs$. For $n\leq 2$, the same representation holds if we replace rotation equivariance with $\On$ equivariance.
    \end{theorem}
    \begin{proof}
        The first part of the statement immediately follows from Theorem~\ref{thm:char_top_degree_real_val}, applied coordinate-wise. Next, for $\xi\in C_b((0,\infty))$ with $\lim_{t\to 0^+} \xi(t)t=0$, we have already seen in Theorem~\ref{thm:class_cov_simple_fconvs} that $\ot_{n,\xi}$ is a continuous, epi-translation invariant valuation with the desired equivariance properties. Furthermore, Theorem~\ref{thm:char_top_degree_real_val} shows that $\ot_{n,\xi}$ is epi-homogeneous of degree $n$.
        
        If now $\oz\colon\fconvs\to\Rn$ is a continuous, epi-translation invariant, rotation equivariant valuation that is epi-homogeneous of degree $n$, then Theorem~\ref{thm:char_top_degree_real_val} shows that the $i$th coordinate of $\oz$ must be of the form
        \[
        u\mapsto \int_{\dom(u)} \zeta_i(\nabla u(x)) \d x
        \]
        for some $\zeta_i\in C_c(\Rn)$, $i\in\{0,\ldots,n\}$. Exploiting the rotation equivariance of $\oz$ (or $\On$ equivariance for $n\leq 2$), as in the proof of Theorem~\ref{thm:class_cov_simple_fconvs}, shows that $\oz=\ot_{n,\xi}$ for some $\xi\in C_b((0,\infty))$ such that $\lim_{t\to 0^+} \xi(t)t=0$.
    \end{proof}

    Before we continue with the translation covariant valuation, we state the dual result to Theorem~\ref{thm:t_alpha_fconvs} on $\fconvf$.
    
    \begin{theorem}
        \label{thm:t_alpha}
        A map $\oz\colon\fconvf\to\Rn$ is a continuous, dually epi-translation invariant valuation that is homogeneous of degree $n$, if and only if there exists a function $\psi\in C_c(\Rn;\Rn)$ such that
        \[
        \oz(v)=\int_{\Rn} \psi(x)\d\MA(v;x)
        \]
        for every $v\in\fconvf$. For $n\geq 3$, the valuation $\oz$ is in addition rotation equivariant, if and only if there exists a function $\xi\in C_b((0,\infty))$ with $\lim_{t\to 0^+} \xi(t)t=0$ such that
        \[
        \oz(v)=\ot_{n,\xi}^*(v)
        \]
        for every $v\in\fconvf$. For $n\leq 2$, the same representation holds if we replace rotation equivariance with $\On$ equivariance.
    \end{theorem}
	
    At last, we look at translation covariant valuations that are epi-homogeneous of degree $(n+1)$. We remark that the following result can be independently obtained from \cite[Theorem 1.9]{knoerr_ulivelli_polynomial}, which uses a similar strategy as the short and direct proof that is presented here.
    
    \begin{theorem}
        \label{thm:class_cov_top_degree_fconvs}
        A map $\oz\colon\fconvs\to\Rn$ is a continuous, translation covariant, vertically translation invariant valuation that is epi-homogeneous of degree $(n+1)$, if and only if there exists a function $\zeta\in C_c(\Rn)$ such that
        \begin{equation}
            \label{eq:oz_top_degree_fconvs}
            \oz(u)=\int_{\dom(u)} \zeta(\nabla u(x))x \d x
        \end{equation}
        for every $u\in\fconvs$. For $n\geq 2$, the valuation $\oz$ is in addition rotation equivariant, if and only if there exists a function $\alpha\in C_c({[0,\infty)})$ such that
        \[
        \oz(u)=\om_\alpha(u)
        \]
        for every $u\in\fconvs$. For $n = 1$, the same representation holds if we replace rotation equivariance with reflection equivariance.
    \end{theorem}
    \begin{proof}
        For $\zeta\in C_c(\Rn)$ it follows from Proposition~\ref{prop:top_degree_val_fconvs} that \eqref{eq:oz_top_degree_fconvs}	defines continuous, translation covariant, vertically translation invariant valuation. Now for $u\in\fconvs$ and $\lambda > 0$ we have $\nabla (\lambda \sq u)(x) = \nabla u(x/\lambda)$ for $x\in\Rn$, and therefore
        \[
            \oz(\lambda \sq u)=\int_{\dom(\lambda \sq u)} \zeta(\nabla u(x/\lambda)) x \d x=\lambda^{n+1} \int_{\dom(u)} \zeta(\nabla u(x))x \d x = \lambda^{n+1} \oz(u),
        \]
        which means that $\oz$ is epi-homogeneous of degree $(n+1)$.
        
        Conversely, let $\oz\colon \fconvs\to\Rn$ be a continuous, translation covariant, vertically translation invariant valuation that is epi-homogeneous of degree $(n+1)$. By Lemma~\ref{le:properties_oz_0} and Lemma~\ref{le:oz_0_hom} together with Theorem~\ref{thm:char_top_degree_real_val}, there exists $\zeta\in C_c(\Rn)$ such that the associated map $\oz^0$ is of the form
        \[
        \oz^0(u)=\int_{\dom(u)} \zeta(\nabla u(x))\d x
        \]
        for $u\in\fconvs$. We now define $\tilde{\oz}\colon\fconvs\to\Rn$ as
        \[
        \tilde{\oz}(u)=\oz(u)-\int_{\dom(u)} \zeta(\nabla u(x))x \d x
        \]
        for $u\in\fconvs$. For any $y\in\Rn$ we have
        \begin{align*}
            \tilde{\oz}(u\circ &\tau_{y}^{-1})\\
            &= \oz(u)+\left(\int_{\dom(u)} \zeta(\nabla u(x))\d x\right) y-\int_{\dom(u)} \zeta(\nabla u(x))x \d x - \left(\int_{\dom(u)} \zeta(\nabla u(x))\d x\right) y\\
            &=\tilde{\oz}(u),
        \end{align*}
        which together with Proposition~\ref{prop:top_degree_val_fconvs} and the properties of $\oz$ shows that $\tilde{\oz}$ is a continuous, epi-translation invariant valuation on $\fconvs$ that is epi-homogeneous of degree $(n+1)$. Theorem~\ref{thm:mcmullen_fconvf} now implies that $\tilde{\oz}$ vanishes identically and hence $\oz$ must be as in \eqref{eq:oz_top_degree_fconvs}.
        
        In the rotation equivariant case, we have already shown the necessary properties of $\om_{\alpha}$ in Theorem~\ref{thm:class_cov_simple_fconvs}. Conversely, if a valuation $\oz$ of the form \eqref{eq:oz_top_degree_fconvs} is rotation equivariant, then, as in the proof of Theorem~\ref{thm:class_cov_simple_fconvs}, it follows from Lemma~\ref{le:g_commutes} that $\oz=\om_{\alpha}$ for some $\alpha\in C_c(\Rn)$.
    \end{proof}
    
    We conclude with the dual result for valuations on $\fconvf$.
	
    \begin{theorem}
        \label{thm:class_cov_top_degree}
        A map $\oz\colon\fconvf\to\Rn$ is a continuous, dually translation covariant, vertically translation invariant valuation that is homogeneous of degree $(n+1)$, if and only if there exists a function $\zeta\in C_c(\Rn)$ such that
        \begin{equation*}
            \oz(v)=\int_{\Rn\times\Rn} \zeta(x)y \d \Theta_0(v;(x,y))
        \end{equation*}
        for every $v\in\fconvf$. For $n\geq 2$, the valuation $\oz$ is in addition rotation equivariant, if and only if there exists a function $\alpha\in C_c({[0,\infty)})$ such that
        \[
        \oz(v)=\om_\alpha^*(v)
        \]
        for every $v\in\fconvf$. For $n = 1$, the same representation holds if we replace rotation equivariance with reflection equivariance.
    \end{theorem}
    
    \subsection*{Acknowledgments}
    This project was supported by the Austrian Science Fund (FWF) 10.55776/P36210.
    

    \vfill
    
    \parbox[t]{8.5cm}{
        Mohamed A. Mouamine\\
        Institute of Discrete Mathematics and Geometry\\
        TU Wien\\
        Wiedner Hauptstra{\ss}e 8-10/104-06\\
        1040 Wien, Austria\\
        e-mail: mohamed.mouamine@tuwien.ac.at
    }
    
    \bigskip
    
    \parbox[t]{8.5cm}{
        Fabian Mussnig\\
        Institute of Discrete Mathematics and Geometry\\
        TU Wien\\
        Wiedner Hauptstra{\ss}e 8-10/104-06\\
        1040 Wien, Austria\\
        e-mail: fabian.mussnig@tuwien.ac.at
    }
\end{document}